\documentclass[11pt]{amsart}

\usepackage{fullpage}
\usepackage{setspace}
\usepackage{amsmath, amssymb, amsfonts, amsthm, mathtools}
\usepackage{stmaryrd}
\usepackage{hyperref}
\usepackage{nameref}
\usepackage[noabbrev,capitalize]{cleveref}

\usepackage{tabularx}
\usepackage{booktabs}
\usepackage{caption}
\usepackage{aurical}

\usepackage{stackengine}

\usepackage{xypic}
\xyoption{all}
\input{xypic}
\newdir{ >}{{}*!/-8pt/\dir{>}}

\DeclareMathSymbol{\mathinvertedexclamationmark}{\mathord}{operators}{'074}
\DeclareMathSymbol{\mathexclamationmark}{\mathord}{operators}{'041}
\makeatletter
\newcommand{\raisedmathinvertedexclamationmark}{%
  \mathord{\mathpalette\raised@mathinvertedexclamationmark\relax}%
}
\newcommand{\raised@mathinvertedexclamationmark}[2]{%
  \raisebox{\depth}{$\m@th#1\mathinvertedexclamationmark$}%
}
\makeatother

\newcommand{\initial}{\raisedmathinvertedexclamationmark}

\newtheorem{thm}{Theorem}[section]
\newtheorem*{thm*}{Theorem}
\newtheorem{lem}[thm]{Lemma}
\newtheorem{cor}[thm]{Corollary}
\newtheorem{prop}[thm]{Proposition}

\newtheorem{Hypotheses}[thm]{Hypotheses}
\newtheorem*{prob*}{Open problem}

\theoremstyle{definition}
\newtheorem{ex}[thm]{Example}
\newtheorem{defi}[thm]{Definition}
\newtheorem{definition}[thm]{Definition}
\newtheorem{notation}[thm]{Notation}

\theoremstyle{remark}

\newtheorem{rem}[thm]{Remark}
\newtheorem*{rem*}{Remark}

\DeclareMathOperator{\id}{id}

\DeclareMathOperator{\Hom}{Hom}

\newcommand{\kringel}{\mathbin{\raise0.5pt\hbox{$\scriptstyle\circ$}}}
\newcommand{\pkt}{\mathbin{\raise0.5pt\hbox{$\scriptstyle\bullet$}}}
\newcommand{\sq}{\mathbin{\raise0.5pt\hbox{$\scriptscriptstyle\square$}}}

\newcommand{\N}{\mathbb{N}}

\renewcommand{\S}{\mathbb{S}}

\newcommand{\End}{{\rm End}}

\newcommand{\CC}{\mathcal{C}}

\newcommand{\CP}{\mathcal{P}}
\newcommand{\CF}{\mathcal{F}}

\newcommand{\coker}{{\rm coker}}
\newcommand{\crmod}{{\rm CrMod}}

\renewcommand{\phi}{\varphi}

\mathchardef\myhyphen="2D 

\usepackage{mathrsfs}
\renewcommand{\CP}{\mathscr{P}}
\renewcommand{\CC}{\mathscr{C}}

\usepackage{enumitem}
\setlist[enumerate,1]{
    label=(\alph*),
    leftmargin =3em
}
\setlist[enumerate,2]{
    label=\roman*),
    leftmargin =3em
}
\setlist[itemize,1]{
    label=---,
    leftmargin =2em
}

\usepackage{tikz,tikz-cd}
\usetikzlibrary{
	decorations.pathmorphing,
	arrows,
	arrows.meta,
	calc,
	shapes,
	shapes.geometric,
	decorations.pathreplacing}
\tikzcdset{arrow style=tikz, diagrams={>=stealth}}
\tikzset{>=stealth'}

\usepackage{xparse}       	
\ExplSyntaxOn
\NewDocumentCommand{\createbunch}{ m O{} m }
 {
  \clist_map_inline:nn { #3 } { \cs_new_protected:cpn { #2 ##1 } { #1 { ##1 } } }
 }
\ExplSyntaxOff
\createbunch{\mathbb}   [bb]{A,B,C,D,E,F,G,H,I,J,K,L,M,N,O,P,Q,R,S,T,U,V,W,X,Y,Z}
\createbunch{\mathcal}  [cal]{A,B,C,D,E,F,G,H,I,J,K,L,M,N,O,P,Q,R,S,T,U,V,W,X,Y,Z}
\createbunch{\mathscr}  [scr]{A,B,C,D,E,F,G,H,I,J,K,L,M,N,O,P,Q,R,S,T,U,V,W,X,Y,Z}
\createbunch{\mathsf}   [sf]{A,B,C,D,E,F,G,H,I,J,K,L,M,N,O,P,Q,R,S,T,U,V,W,X,Y,Z}
\createbunch{\mathrm}   [rm]{A,B,C,D,E,F,G,H,I,J,K,L,M,N,O,P,Q,R,S,T,U,V,W,X,Y,Z,a,b,c,d,e,f,g,h,i,j,k,l,m,o,p,q,r,s,t,u,v,w,x,y,z}
\createbunch{\mathfrak} [frak]{A,B,C,D,E,F,G,H,I,J,K,L,M,N,O,P,Q,R,S,T,U,V,W,X,Y,Z}
\createbunch{\mathfrak} [frak]{a,b,c,d,e,f,g,h,i,j,k,l,m,n,o,p,q,r,s,t,u,v,w,x,y,z}

\newcommand{\Lie}{\mathcal{L}ie}
\newcommand{\As}{\calA s}
\newcommand{\susp}{\rms}
\newcommand{\asusp}{\rms^{-1}}

\DeclareMathOperator{\Ker}{\mathrm{ker}}

\renewcommand{\Bar}{\operatorname{B}}

\newcommand{\too}{\longrightarrow}
\usepackage{mathbbol}
\newcommand{\field}{\mathbb{k}}

\renewcommand{\geq}{\geqslant}
\renewcommand{\leq}{\leqslant}

\begin{document}

\title{Crossed modules and cohomology of algebras over an operad}

\author{Johan Leray}
\author{Salim Rivi\`ere}
\author{Friedrich Wagemann}

\date{\today}

\subjclass[2020]{}
\keywords{}

\begin{abstract}
  We introduce a general definition of a $n$-crossed module of $\scrP$-algebras over an algebraic operad $\CP$, which coincides with historical definitions in the cases of the operads $\As$ of $\Lie$ and $n = 1$. We establish a natural isomorphism between the abelian group of equivalence classes of $n$-crossed modules over a pair $(A,M)$ for an operad $\CP$ and the $(n+1)^\text{st }$ operadic cohomology group of $A$ with coefficients in $M$.
\end{abstract}

\maketitle


\section*{Introduction}

Crossed modules of groups appeared first in work of J. H. C. Whitehead \cite{W} on the relative homotopy groups of the attachment of cells to a $1$-skeleton. A {\it crossed module of groups} is a morphism of groups $\mu:M\to N$ together with an action of $N$ on $M$ by automorphisms such that $\mu$ is equivariant and such that the adjoint action on $M$ can be obtained by sending elements of $M$ to $N$ via $\mu$ and letting them then act again on $M$. Mac Lane and Whitehead \cite{MLW} established a link between crossed modules and degree $3$ cohomology classes. According to Mac Lane's historical note \cite{ML}, Gerstenhaber \cite{G1}, \cite{G2}  was the   first to establish the theorem that degree $3$ group cohomology corresponds bijectively to equivalence classes of crossed modules. Moreover, Gerstenhaber extends this kind of theorem to more general algebraic structures including the theorems that equivalence classes of crossed modules of associative algebras correspond bijectively to the third Hochschild cohomology group, and that equivalence classes of crossed modules of Lie algebras correspond to the third Chevalley-Eilenberg cohomology group. However, fundamental work on crossed modules was also performed not much later by Brown \cite{B}, Huebschmann \cite{H}, Duskin \cite{D}, Ratcliffe \cite{R} and many others. The interested reader may consult these original articles or the recent monograph \cite{Wa}.   

The article \cite{Baues:2004} of Baues, Minian and Richter adresses a generalization of these questions to the operadic setting. In fact, they describe crossed modules of $\CP$-algebras as algebras over a secondary operad associated to $\CP$ and focus then on a cohomology theory built from splicing together long exact sequences with crossed modules. The relation of this cohomology theory to Hochschild theory in the case of associative algebras and to Harrison cohomology in the case of associative commutative algebras is established, but the link to general operadic cohomology is missing. In the present article, our goal is to link crossed modules of $\CP$-algebras to operadic cohomology in order to show that equivalence classes of crossed modules correspond bijectively to cohomology classes.    

For this, we present first of all a general definition (slightly different from the one of Baues, Minian and Richter in \cite{Baues:2004}) of a crossed module of $\CP$-algebras. 
Namely, an $n$-crossed module of $\CP$-algebras is simply a $\CP$-algebra structure on a length $n$ complex of vector spaces (\emph{i.e.} having possibly $n+1$ non-zero terms). 
We show that crossed modules of associative algebras and crossed modules of Lie algebras fit into this setting. 
Then we go on establishing a link to the operadic cohomology, denoted for a $\CP$-algebra $A$ and an $A$-module $M$ by $\rmH^{m}_\CP(A;M)$. 
Namely, we show that the Homotopy Transfer Theorem (HTT) can be used to associate to a crossed module a cocycle and the Rectification Theorem (RT) can be used in order to associate a crossed module to a cocycle. These two theorems are well known in operad theory (see \cite[Theorem 10.3.2 and Theorem 11.4.7]{Loday:2012}). The HTT describes how the $\CP$-algebra structure on a dg $\CP$-algebra $(A,d)$ can be transferred into a $\CP_\infty$-structure on a homotopy retract of the complex $(A,d)$. In our framework, we will transfer the $\CP$-algebra structure of an $n$-crossed module with "kernel" $M$ and "cokernel" $A$ to a $\CP_\infty$-structure on $A\oplus M[n]$. This describes it as a cocycle in the appropriate setting. 

The second step (in fact, the return step) is the Rectification Theorem. The RT shows how a $\CP_\infty$-algebra structure may be rectified to a $\CP$-algebra structure on a much bigger complex (passage from $A$ to $\Omega B A$). We show how to truncate this much bigger complex to a quasi-isomorphic length $n$ complex (which is still a $\CP$-algebra) when starting with a cocycle. 

Denote for an integer $n\geq 1$, a $\CP$-algebra $A$ and an $A$-module $M$ by $\crmod^n_\CP(A,M)$ the set of $n$-crossed modules over $A\oplus M[n]$. It carries an equivalence relation (see Definition \ref{definition_equivalence}) with associated set of equivalence classes  denoted by
$$\crmod^n_\CP(A,M)/\sim.$$ 
We show in Section 4 that there is a Baer-sum on crossed modules yielding an abelian group structure on $\crmod^n_\CP(A,M)/\sim$. 

The main theorem reads:

\begin{thm*}[cf. \cref{main_theorem}]
  Let $\CP$ be an augmented operad concentrated in degree zero, let $n$ be a positive integer, let $A$ and $M$ be respectively a $\CP$-algebra and an $A$-module, both concentrated in degree $0$. 
  There is a natural isomorphism of abelian groups
  \[ \crmod^n_\CP(A,M)/\sim\quad \cong\,\,\rmH^{n+1}_\CP(A;M). \]
\end{thm*}


While the HTT works in a general properadic setting (see \cite{HLV}), the RT works, as far as we know, only for operads. Therefore only one half of our theorem may be generalized to properads, namely there exists a well-defined natural map
  \[ \crmod^n_\CP(A,M)/\sim\quad \too\,\,\rmH^{n+1}_\CP(A;M). \]
It would be interesting to investigate the classes in Gerstenhaber-Schack cohomology which one obtains in this way from crossed modules of associative bialgebras.


  %

\subsection*{Conventions}\label{Conventions}
We work over a field $\field$ of characteristic $0$
We adopt the {\it homological convention}, i.e. dg vector spaces have a differential which lowers degree. We denote by $\susp$ the dg vector space given by the ground field $\field$ concentrated in degree $1$, and $\asusp \coloneqq \Hom_k(\susp,\field)$ which is concentrated in degree $-1$. So we have the following Koszul sign rule $\asusp\susp = 1 = -\susp\asusp$. For a dg vector space $M$, we will use the notation $M[n]$ for the suspension  $\susp^n\otimes M$.

\subsection*{Notations}


Note that we do not write the composition of maps with a sign $\circ$. This sign is reserved for the circle product of ${\mathbb S}$-modules. Instead, we just put a little space between the maps.

\subsection*{Aknowledgement} The authors warmly thank Geoffrey Powell for his careful reading and many comments on the first version of this paper.


\section{Preliminairies on operads and their algebras}

Let us recall the parts of the general framework of operads and their algebras that we need for our article. The main reference is  \cite{Loday:2012}. Recall that an $\mathbb{S}$-module $M = \{M(n), n \in\N\}$ is a collection of dg vector spaces indexed by integers equipped with an action of the symmetric group $\mathbb{S}_n \times M(n) \to M(n)$. To an $\bbS$-module $\CP$, one associate the endofunctor of the category of dg vector spaces $A \mapsto \CP(A) = \bigoplus_n \CP(n) \otimes_{\bbS_n} A^{\otimes n}$ called the \emph{Schur functor}. The Schur functors $A\mapsto\CP(A)$ admits also linear and $N$-multilinear versions where one considers only linear operations in one slot. 
The functor $A\mapsto\CP(A)$ is in general not additive, but analytic, i.e. there exists a decomposition
\[
\CP(A\oplus N)=\CP(A)\oplus\CP(A;N)\oplus\CP(A\oplus N)^{(\geqslant 2)_N}.
\]
Here $\CP(A;N)$ is the space of $\CP$-operations on $A\oplus N$ with exactly one entry in $N$, and $\CP(A\oplus N)^{(\geqslant 2)_N}$ denotes all the remaining parts, i.e. the $\CP$-operations which have two and more entries in $N$ and which are thus $2$- or more $N$-linear. Let us define in general $\CP(A \oplus N)^{(l)_N}$ to be the $l$-linear part in $N$, i.e. the $\CP$-operations on $A\oplus N$ with exactly $l$ entries in $N$. We have thus by definition $\CP(A \oplus N)^{(1)_N} = \CP(A; N)$.

\subsection*{Operads and cooperads} By definition, a {\it symmetric operad}  $\CP$ is a monoid in the monoidal category $(\S\myhyphen\mathsf{mod}, \circ, \mathcal{I})$, see  \cite[Sect. 5.2.1]{Loday:2012}, where $\circ$ is the composition product of $\S$-modules, see Sect. 5.1.4 in {\it loc. cit.}. Such monoid structure on $\CP$ is equivalent to a monad structure on the associated Schur functor (see \cite[Sect. 5.1.2]{Loday:2012}). The symbol $\CF(M)$ denotes the {\it free operad} on the $\S$-module $M$, see \cite[Sect. 5.5]{Loday:2012} which is described in terms of rooted trees. Cooperads are defined similarly as comonoids, see Sect. 5.8 in {\it loc. cit.}. There is a similar functor $\CF^c$ defining the cofree cooperad, Sect. 5.8.6 in {\it loc. cit.}. 
The notions of connected (se  \cite[Sect. 6.3.11]{Loday:2012}) coaugmented (see \cite[Sect. 5.8.1]{Loday:2012}) and conilpotent (see  \cite[Sect. 5.8.5]{Loday:2012})cooperad  are transferred of theory of coalgebras.


The linear parts in ${M}$ and ${N}$ of ${M}\circ{N}$ define the {\it infinitesimal composition product} ${M}\circ_{(1)}{N}$, see Sect. 6.1.1 in {\it loc. cit.}. Given an operad $(\CP,\gamma,\eta)$ and a cooperad $(\CC,\triangle,\epsilon)$, the {\it convolution product} of two elements $f,g\in\Hom_\S(\CC,\CP)$ (see \cite[Prop. 6.4.3.]{Loday:2012})  is defined by the composition
\[
  f\star g:\CC
  \xrightarrow{\triangle_{(1)}}
  \CC\circ_{(1)}\CC
  \xrightarrow{f\circ_{(1)} g}
  \CP\circ_{(1)}\CP
  \xrightarrow{\gamma_{(1)}}
  \CP.
\]

\subsection*{Bar-cobar adjunction for operads} A {\it twisting morphism} $\alpha\in\Hom_\S(\CC,\CP)$ is a degree $-1$ solution to the Maurer-Cartan equation 
\[
  \partial(\alpha)+\alpha\star\alpha = 0,
\]
where $\partial$ is the differential on the Hom-space induced by the differentials on $\CC$ and $\CP$.

The {\it bar construction} (Sect. 6.5.1 in {\it loc. cit.}) of an augmented dg operad $(\CP,d_\CP,\gamma,\eta,\epsilon)$ is the quasi-cofree conilpotent dg cooperad 
\[
  \Bar\CP\coloneqq (\CF^c(\susp\overline{\CP}),d_1+d_2),
\] 
where $\overline{\CP}$ is the kernel of the augmentation, a.k.a. the augmentation ideal, $d_1$ is the unique coderivation extending the internal differential of $\CP$ and $d_2$ is the unique coderivation extending the infinitesimal composition map $\gamma_{(1)}$. The symbol $\susp$ denotes the degree shift (see the \nameref{Conventions}). Similarly, the {\it cobar construction} (Sect 6.5.2 in {\it loc. cit.})
associates to a coaugmented dg cooperad $(\CC,d_\CC,\triangle,\epsilon,\eta)$ the augmented quasi-free operad 
\[
\Omega\CC\coloneqq (\CF(s^{-1}\overline{\CC}),d_1+d_2),
 \]
where again $\overline{\CC}$ is the cokernel of the coaugmentation,  $d_1$ stems from the internal differential and $d_2$ from the infinitesimal decomposition map $\triangle_{(1)}$. The \emph{Rosetta stone} relates twisting morphisms and maps from resp. to the cobar and bar construction:
 
\begin{prop}[{\cite[Theorem 6.5.7]{Loday:2012}}]  
For every augmented operad $\scrP$ and every conilpotent cooperad $\scrC$, there exist natural isomorphisms
\[
\Hom_{\mathsf{operads^{\rm aug}}}(\Omega\CC,\CP)\cong{\rm Tw}(\CC,\CP)\cong \Hom_{\mathsf{cooperads^{\rm coaug}}}(\CC,\Bar\CP),
\] 
In particular, the functors $\Omega$ and $\Bar$ are adjoint.
\end{prop} 
  
\noindent In the model structure on the category of operads (see \cite[Erratum Thm 3.3]{Hinich}) where weak-equivalence are componentwise quasi-isomorphisms and fibrations are componentwise epimorphisms, the counit of the adjunction provides a functorial cofibrant resolution $\Omega \Bar\CP\to\CP$ for any operad $\CP$. More generally, a twisting morphism $\alpha : \CC \to \CP$ which induced a quasi-isomorphism $\Omega\CC\overset{\sim}{\to} \CP$ is called a \emph{Koszul morphism}.

\begin{center}
  \emph{In the rest of this section, we fix $\alpha : \CC \to \CP$ a Kozsul morphism.}
\end{center}

\subsection*{Algebras over an operad} Operads are universal objects which encode algebraic structures. For an operad $\CP$, a $\CP$-algebra $A$ is a dg vector space equiped with a morphism of operads $\mu_A : \CP \to \End_A$, where for $n\in\N$, $\End_A(n) = \Hom(A^{\otimes n}, A)$. Such structural morphism is equivalent to a morphism of dg vector spaces $\CP(A) \to A$. More generally, for any Kozsul morphism $\alpha : \CC \to \CP$, we will call $\Omega\CC$-algebras {\it $\CP_\infty$-algebras} or {\it homotopy $\CP$-algebras}. 

\begin{prop}[The Rosetta Stone {\cite[Theorem 10.1.13]{Loday:2012}}]  \label{prop_rosetta_stone}
The set of $\CP_\infty$-algebra structures on a dg vector space $A$ is equivalently given by
\[
\Hom_{\mathsf{operads}}(\Omega\CC,\End_A)\cong{\rm Tw}(\CC,\End_A)\cong \Hom_{\mathsf{cooperads}}(\CC,\Bar\End_A) \cong \mathrm{Codiff}(\CC(A)),
\] 
where $\mathrm{Codiff}(\CC(A))$ is the set of codifferentials, i.e. the set of degree $-1$ square-zero coderivation on the $\CC$-coalgebra $\CC(A)$.
\end{prop} 

\noindent One can associate to any twisting morphism $\alpha:\CC\to\CP$ two functors 
\[
\Bar_\alpha:\CP\myhyphen{\sf Alg} \to \CC\myhyphen{\sf Coalg}^{\rm conilp.},\,{\rm and}\,\,
\Omega_\alpha:\,\CC\myhyphen{\sf Coalg}^{\rm conilp.}\to\CP\myhyphen{\sf Alg},
\]
see \cite[Sect. 11.2.2 and 11.2.5]{Loday:2012}. By definition, the {\it bar construction with respect to $\alpha$},  denoted $\Bar_\alpha A=\CC(A)$, is the cofree conilpotent $\CC$-coalgebra on the $\scrP$-algebra $A$ with a square zero coderivation $d_1+d_2$, is obtained from the differential of $\CC$ and the differential of $A$ (corresponding to $d_1$) and the morphism $\alpha$ (corresponding to $d_2$). The {\it cobar construction with respect to $\alpha$}, denoted $\Omega_\alpha(C)$, is defined similarly on $\CP(C)$, the free $\CP$-algebra on the vector space $C$.


\subsection*{$\infty$-morphisms between operadic algebras} Between two $\Omega\CC$-algebras, one define a nicer notion of morphisms called $\infty$-morphisms. A $\infty$-morphism $\varphi$ between two $\Omega\CC$-algebras $(A, \mu_A : \CC \to \End_A)$ and $(B, \mu_B : \CC \to \End_B)$ is a map of $\mathbb{S}$-modules $\phi : \CC \to \End^A_B$, where $\End^A_B(n) = \Hom(A^{\otimes n}, B)$ satisfying the equation 
\[
  \partial \varphi = \varphi \rhd \mu_A - \mu_B \lhd \varphi
\]
where we have
\[
  \begin{aligned}
  \varphi \rhd \mu_A & : \CC \xrightarrow[]{\triangle_{(1)}} \CC \circ_{(1)} \CC \xrightarrow[]{\varphi \circ \mu_A} \End_B^A \circ \End_A \xrightarrow[]{\mathrm{comp.}} \End_B^A, \\
  \mbox{ and } \mu_B \lhd \varphi & : \CC \xrightarrow[]{\triangle} \CC \circ \CC \xrightarrow[]{\mu_B \circ \varphi} \End_B \circ \End_B^A \xrightarrow[]{\mathrm{comp.}} \End_B^A \\
  \end{aligned}
\]
By interpreting $\CP_\infty$-structure on $A$ as a codifferential on $\CC(A)$ (cf. \Cref{prop_rosetta_stone}), an $\infty$-morphism $A \leadsto B $ corresponds to a strict morphism of $\CC(A) \to \CC(B)$ of codifferential $\CC$-coalgebras. 

Notice that the restriction of an $\infty$-morphism $\varphi : A \leadsto B$ to the part $A \cong \mathcal{I}(A) \subset \CC(A)$ is a chain map : 
\[
  \begin{tikzcd}
    \mathcal{I}(A) \arrow[r] \arrow[d, "\varphi_0"'] & \CC(A) \arrow[d, "\varphi"] \\
    \mathcal{I}(B) \arrow[r] & \CC(B) \\
  \end{tikzcd}  
\]
We say that $\varphi$ is a $\infty$-quasi-isomorphism if the component $\phi_0 : A \to B$ is a quasi-isomorphism of chain complexes. For more details, see \cite[Sect. 10.2]{Loday:2012}.

\subsection*{Homotopy Transfer Theorem} Let $(A,d_A)$ be a chain complex which admits a {\it contraction} onto another chain complex $(H,d_H)$, i.e. there exist chain maps $i$ and $p$ and a homotopy $h$ of degree $1$ which satisfy
\[
p\,i=\id_H\quad i\,p-\id_A=d_A\,h+h\,d_A.
\] 
In this situation, there is a Homotopy Transfer Theorem (HTT), see  \cite[Sect. 10.3]{Loday:2012}. The HTT transfers a $\CP_\infty$-structure on $A$ to a deformation retract $H$. 

\begin{thm}[Homotopy Transfer Theorem] \label{HTT}
Given an $\Omega\CC$-algebra structure $\alpha\in{\rm Tw}(\CC,\End_A)$ on $A$, then there is an $\Omega\CC$-algebra structure on $H$ and two $\infty$-quasi-isomorphisms $i_\infty$ and $p_\infty$ extending 
the maps $i:H\to A$ and $p:A\to H$. 
\end{thm}

The $\infty$-morphisms $i_\infty$ and $p_\infty$ have an explicit combinatorial description in terms of trees labeled by $i$, $p$ and $h$ (see \cite[Sect. 10.3]{Loday:2012} or \cite{HLV}), 

\subsection*{Rectification theorem} The rectification of a $\CP_\infty$-structure into a $\CP$-structure is given by the Rectification Theorem, see \cite[Theorem 11.4.4]{Loday:2012}.
We state it here in the general framework of an operad $\CP$  
such that there exists a coaugmented conilpotent cooperad $\CC$ 
and a twisting morphism $\alpha:\CC\to \CP$ inducing a quasi-isomorphism $\Omega\CC\to\CP$ (see Proposition \ref{prop_rosetta_stone}). These hypotheses are fulfilled, for example, for $\CC=\Bar\CP$, but also for a Koszul operad $\CP$ and $\CC=\CP^{\initial}$. Note that in terms of the notation used in \cite{Loday:2012}, we have in our situation $\kappa=\alpha:\CC\to \CP$ and $\iota:\CC\to\Omega\CC$ the universal twisting morphism (cf \cite[Sect. 2.2.4]{Loday:2012}). 

\begin{thm}[Rectification Theorem] \label{RT}
In the above stated situation, any homotopy $\CP$-algebra $A$ is naturally $\infty$-quasi-isomorphic to the $\CP$-algebra $\Omega_\alpha \Bar_\iota A$. 
\end{thm}  

\begin{proof} 
Can be adapted from the proof of \cite[Theorem 11.4.4]{Loday:2012}.
\end{proof} 

\subsection*{Operadic cohomology}
For an operad $\CP$  such that there exists a coaugmented conilpotent cooperad $\CC$ and a twisting morphism $\alpha:\CC\to \CP$ inducing a quasi-isomorphism $\Omega\CC\to\CP$, let us recall the {\it operadic cohomology} of a $\CP$-algebra $A$ with values in an $A$-module $(M,\gamma_M)$, see \cite[Sect. 12.4]{Loday:2012}. It is the homology of the chain complex underlying the bar construction of $A$ with respect to $\alpha$:
\[
\rmC^\bullet_\CP(A,M)\coloneqq (\Hom(\Bar_\alpha A,M),\partial_{\pi_\alpha})=(\Hom(\CC(A),M),\partial_{\pi_\alpha}).
\] 
Denoting $d$ the differential on the bar construction $\Bar_\alpha A$ which is given by the sum of a differential $d_1$ incorporating the differential of $A$ and the differential of $\CC$ and a differential $d_2$ induced from the twisting morphism $\alpha$, the differential $\partial_{\pi_\alpha}$ is given by 
\begin{equation}  \label{definition_differential}
\partial_{\pi_\alpha}(g)\coloneqq \bar{\partial}(g)-(-1)^{|g|}g\,d,
\end{equation}
where $g\in\Hom(\CC(A),M)$ and where $\bar{\partial}$ is the composition
\[
  \CC\xrightarrow{\triangle}
  \CC\circ\CC(A)\xrightarrow{\alpha\circ({\rm proj};g)}
  \CP\circ(A;M)\xrightarrow{\gamma_M} M. 
\]
Here ${\rm proj}:\CC(A)\to A$ denotes the canonical projection. As mentioned in \cite[Sect.12.4]{Loday:2012}, this operadic cohomology gives back the usual cohomology theories of Hochschild cohomology for the associative operad $\As$ and Chevalley--Eilenberg cohomology for the Lie operad $\Lie$ up to a degree shift. 
As usual, we denote by $\rmH_\CP^\bullet(A;M)$ the cohomology groups of the $\CP$-algebra $A$ with values in the $A$-module $M$, and by  $\rmZ_\CP^n(A;M)\coloneqq \ker\left(\partial_{\pi_\alpha}\colon \rmC_\CP^n(A;M)\to \rmC_\CP^{n+1}(A;M)\right)$ the space of $n$-cocycles.

 
\section{Crossed modules of algebras over an operad}

\begin{Hypotheses}\label{HypothesesOperads}
  In this paper, we suppose that 
  \begin{itemize}
    \item $\CP$ is a augmented operad concentrated in degree $0$;
    \item $\CC$ is a connected coaugmented conilpotent cooperad concentrated in positive degrees with a twisting morphism $\alpha : \CC \to \CP$ such that the induced morphism of operads $\Omega\CC \xrightarrow[]{\sim} \CP$ is a cofibrant resolution of $\CP$: such twisting morphisms could be produced by the Bar-Cobar resolution or by Koszul duality for example;
    \item $A$ is a $\CP$-algebra concentrated in degree $0$;
    \item $M$ is a $A$-module concentrated in degree $0$.
  \end{itemize}
\end{Hypotheses}

We denote by $A\ltimes M$ the corresponding $\CP$-algebra structure on the direct sum $A\oplus M$, see \cite[Sect. 12.3.2]{Loday:2012}. 

\begin{definition}[$n$-crossed module]\label{def:ncrossedmodule}
  Let $A$ be a $\CP$-algebra, let $M$ be an $A$-module and $n\geqslant 1$ be an integer. An \emph{$n$-crossed module $B$ over the pair $(A,M)$} is a 
 map of dg $\CP$-algebras $\pi\colon B \to A$ such that 
  \begin{itemize}
    \item $B=(0\to B_n\xrightarrow{\partial_n}\cdots\xrightarrow{\partial_1} B_0\to 0)$ 
    is concentrated in degrees $0,\ldots,n$;
    \item there exists a chain map $\varphi \colon B \dashrightarrow A \ltimes M[n]$ which is a quasi-isomorphism such that the following diagram of chain maps commutes:
    \[\begin{tikzcd}
      B 
      \arrow[rr, dashed, "\simeq", "\varphi"'] 
      \arrow[rd, "\pi"'] 
      & & A \ltimes M[n] \arrow[ld, "\mathrm{pr}_A"] \\
      & A &
    \end{tikzcd} ;\]
    \item the chain map $\Ker\left(\pi \right) \to M$ induced by $\varphi$ is a morphism of $B$-modules (where $M$ is a $B$-module via $\pi$).
  \end{itemize}
 \end{definition}

\begin{rem}\label{rem:descriptionCrossedMod}
The simplest definition of an {\it $n$-crossed module} is just a $\CP$-algebra structure on a complex  
$B=(0\to B_n\xrightarrow{\partial_n}\cdots\xrightarrow{\partial_1} B_0\to 0)$ 
which is concentrated in degrees $0,\ldots,n$, with non trivial homology groups only in degrees $0$ and $n$. 
The extra data in the above definition serves to express the notion of an $n$-crossed module {\it with fixed kernel $M$ and fixed cokernel $A$}. In other words, it is an $n$-crossed modules {\it over the pair (A,M)}. 
\end{rem}

\begin{notation}
Let us denote by $\crmod^n_\CP(A,M)$ the set of $n$-crossed modules over $(A,M)$ (in a fixed universe). 
\end{notation}

\begin{rem} 
  Using the description of crossed modules given in Remark \ref{rem:descriptionCrossedMod}, we have directly that any morphism of operads $\calQ \to \CP$ induces, for all pairs $(A,M)$, a map 
  \[\crmod^n_\CP(A,M) \to \crmod^n_\calQ(A,M).\]
\end{rem}

\begin{definition}[Equivalence of $n$-crossed modules]  \label{definition_equivalence}
Two $n$-crossed modules $B$ and $B'$ over the pair $(A,M)$ are called \emph{elementarily equivalent}
if there exists a quasi-isomorphism of $\CP$-algebras $\psi:B\to B'$ such that $\varphi'\,\psi=\varphi$. The \emph{equivalence of crossed modules} is the equivalence relation generated by elementary equivalence.   
\end{definition}

\begin{rem}
Note that elementary equivalence implies in particular that $\pi'\,\psi=\pi$ as morphisms of $\CP$-algebras and $\varphi'\,\psi=\varphi$ as morphisms of $B$-modules (where $\ker(\pi')$ is a $B$-module via $\psi$). 
\end{rem}

\begin{prop}  \label{proposition_crmod}
Suppose that $n\geqslant 1$. 
The data of an $n$-crossed module over the pair $(A,M)$ is equivalent to the data of a $\CP$-algebra structure on a complex  
$B=(0\to B_n\xrightarrow{\partial_n}\cdots\xrightarrow{\partial_1} B_0\to 0)$ 
together with a quasi-isomorphism $\phi:B\to A\ltimes M[n]$ such that $\rmH_0(\phi):(B_0/\partial_1 B_1)\to A$ is an isomorphism of $\CP$-algebras and $\varphi |_{\Ker\left(\pi \right)}$ is a morphism of $B$-modules. 
Furthermore $\varphi$ induces an isomorphism of $A$-modules $\varphi:B_n\to M$. 
\end{prop}

\begin{proof}
It is clear that an $n$-crossed module induces the data stated in the proposition, because the 
chain map $\pi$ induces the $\CP$-algebra isomorphism $\rmH_0(\phi):(B_0/\partial_1 B_1)\to A$ and the isomorphism of $A$-modules $\Ker\left(\rmH_*(\pi)\right) \to M$  (where the $\rmH_0(B)$-module $\Ker\left(\rmH_*(\pi)\right)$ is viewed as an $A$-module via the isomorphism $\rmH_0(\phi):(B_0/\partial_1 B_1)\to A$). 

Conversely, 
 we get back the data of an $n$-crossed module with $\pi$ being the projection $B\to B_0\to(B_0/\partial_1 B_1)$ followed by the isomorphism $\rmH_0(\phi)$ (this is a chain map !). 
 
The map $\varphi$ has non zero components only in degree $0$ and $n$. Its degree-$n$-component $\varphi_n:B_n\to M[n]$ is a morphism of $B_0$-modules for degree reasons. 

Also for degree reasons, $\ker(\partial_n)$ is a $B_0$-module with respect to the $\CP$-algebra structure $\mu$. As $\mu$ is a chain map, the action of elements $b_1,\ldots,b_k \in B_0$ where at least one element is in $\partial_1 B_1$, is zero on $\ker(\partial_n)$ : suppose that $b_i = \partial_1 \tilde{b}$, and $m \in \ker(\partial_n)$, we have
\[
  \begin{aligned}
    \mu(b_1, \ldots, \partial_1 \tilde{b}, \ldots, b_k, m) & =  
    \partial_{n+1}\mu(b_1, \ldots, \tilde{b}, \ldots, b_k, m) \\
    & \qquad\pm \mu(\partial_0 b_1, \ldots,  \tilde{b}, \ldots, b_k, m) \pm \cdots  \\
    & \qquad \pm \mu(b_1, \ldots,  \tilde{b}, \ldots, \partial_0  b_k, m) \\ 
    & \qquad \pm \mu(b_1, \ldots,  \tilde{b}, \ldots,   b_k,\partial_n m)
  \end{aligned}
\]
where on the right hand side, all the terms are zero. Thus $\ker(\partial_n)$ becomes an $A\cong B_0/\partial_1 B_1$-module.   
\end{proof}

Let us now give some examples of $n$-crossed modules over an operad $\CP$ in the special case $n=1$. Forgetting the identification of the kernel and the cokernel, a crossed module of $\CP$-algebras is just a $\CP$-algebra structure on a length-1-chain complex $(V,\partial)\coloneqq (\partial:B\to C)$. We will establish in particular that this definition generalizes known definitions of crossed modules.  

\begin{ex}
Let us first consider $1$-crossed module of $\As$-algebras. Let $\partial:B\to C$ be a complex. The structure of an  $1$-crossed module of $\As$-algebras on $\partial:B\to C$ is by definition the structure of an associative algebra on $\partial:B\to C$ in the monoidal category of chain complexes. Let us assign to $B$ the degree $1$ and to $C$ the degree $0$. This means then that we have an associative product
\[
(B\to C)\otimes(B\to C)\to (B\to C).
\]
The (degree $0$-component of the) associative product $C\otimes C\to C$ renders $C$ an (ordinary)  associative algebra. This owes to the principle that thanks to the fact that $C\mapsto (0\to C)$ and $(B\to C)\mapsto C$ are monoidal functors, the degree $0$ component of any $\CP$-algebra in complexes acquires the structure of a $\CP$-algebra, cf \cite{LP}. 

The degree-$1$-component of the product $(C\otimes B+B\otimes C)\to B$ gives a $C$-bimodule structure on $B$ (thanks to the associativity of the product).  The commutativity of the diagram
\[
  \begin{tikzcd}
    C\otimes B+B\otimes C \ar[r] \ar[d] & B \ar[d] \\ C\otimes C \ar[r] & C 
  \end{tikzcd}
\]
implies that for all $b,b'\in B$ and all $c,c'\in C$, we have 
\[
c\cdot\partial(b) + \partial(b')\cdot c'= c\partial(b)+\partial(b')c'.
\]
This is the equivariance of the map $\partial:B\to C$. 
The commutativity of the square
\[
  \begin{tikzcd}
    B\otimes B \ar[r] \ar[d] & 0 \ar[d] \\ 
    C\otimes B+B\otimes C \ar[r] & B 
  \end{tikzcd}
\]
induces a relation called the \emph{Peiffer relation.} Indeed, the commutativity implies for all $b,b'\in B$ that
\[
\partial(b)\cdot b' - b\cdot\partial(b') =0,
\]
which gives $\partial(b)\cdot b'=b\cdot\partial(b')$, compare \cite[Definition 3.1]{Baues:2002}. 
Defining a product on $B$ via the prescription $bb'\coloneqq \partial(b)\cdot b'$ renders $B$ an associative algebra such that we obtain a crossed module of associative algebras $\beta:B\to C$ where, with respect to \cref{def:ncrossedmodule}, $A = \coker(\partial)$, $M = \ker(\partial)$ and $\phi$ is the projection.
\end{ex}

\begin{ex}
Let us now consider $1$-crossed modules of $\Lie$-algebras. Let $\partial:B\to C$ be a complex. Again, the structure of a  $1$-crossed module of $\Lie$-algebras on $\partial:B\to C$ is by definition the structure of a Lie algebra on $\partial:B\to C$ in the monoidal category of chain complexes. This means then that we have a bracket
\[
[-,-]:(B\to C)\otimes(B\to C)\to (B\to C)
\]
which must be antisymmetric and must satisfy the Jacobi identity. The (degree $0$-component of the) bracket $[-,-]:C\otimes C\to C$ renders $C$ a(n ordinary) Lie algebra as before.

The degree-$1$-component of the bracket $[-,-]:(C\otimes B+B\otimes C)\to B$ gives two operations $(c,b)\mapsto c\cdot b$ and $(b,c)\mapsto b\cdot c$ on $B$. The antisymmetry of the bracket implies that $c\cdot b=-b\cdot c$ for all $b\in B$ and all $c\in C$, thus we can restrict to the left operation. The Jacobi identity implies that this left operation is indeed a Lie algebra action. The commutativity of the diagram
\[
  \begin{tikzcd}
    C\otimes B \ar[r] \ar[d] & B \ar[d] \\
    C\otimes C \ar[r] & C
  \end{tikzcd}  
\]
implies that for all $b\in B$ and all $c\in C$, we have 
\[
\partial(c\cdot b)= [c,\partial(b)].
\]
This is the equivariance of the map $\partial:B\to C$. 
The commutativity of the square
\[
  \begin{tikzcd}
    B \otimes B \ar[r] \ar[d] & 0 \ar[d] \\
    C\otimes B \oplus B \otimes C \ar[r] & B
  \end{tikzcd}  
\]
induces the relation called the Peiffer relation. 
Indeed, the commutativity implies for all $b,b'\in B$ that
\[
\partial(b)\cdot b' = b\cdot\partial(b').
\]
We thus have $\partial(b)\cdot b' = b\cdot\partial(b') = -\partial(b')\cdot b$. 
Defining a bracket on $B$ via the prescription $[b,b']\coloneqq \partial(b)\cdot b'$ renders this 
bracket antisymmetric. The Jacobi identity follows like in \cite{LP}. So we obtain a crossed module of Lie algebras $\partial:B\to C$.    
\end{ex}


\begin{ex}
  In \cite{Luo}, the author defines the notion of crossed modules of preLie algebras and shows that the sets of equivalent classes of crossed modules extensions of $(\mathfrak{g}, V)$ is in bijection with $\rmH^3(\mathfrak{g}, V)$ where $\mathfrak{g}$ is a preLie algebra and $V$ is a representation of $\mathfrak{g}$ (see \cite[Thm. 3.2]{Luo}). As in the case of associative algebras or Lie algebras, our definition of $1$-crossed modules of preLie algebras over $(\mathfrak{g}, V)$ coincides with the notion of crossed modules extentions of $(\mathfrak{g}, V)$ defined in \emph{loc. cit.}, and \cite[Thm. 3.2]{Luo} can be viewed as an incarnation of \cref{main_theorem} when $\CP$ is the operad of preLie algebras.
\end{ex}

\begin{rem}
  In \cite{CCKL}, the authors construct some functors between categories of $1$-crossed modules of associative, Lie, Leibniz and diassociative algebras. Several of them provide from an underlying morphism of operads.
\end{rem}

\begin{ex}
Let $A$ be a $\CP$-algebra, $M$ be an $A$-module and $n=1$. Then the trivial map $0:M\to A$ is a crossed module of $\CP$-algebras. Moreover, the inclusion map of an ideal $\partial:I\to A$ gives a(n injective) crossed module, and the extensions $\partial:E\to A$ give rise to (surjective) crossed modules. This is analogous to what happens for many algebraic structures.  
%

\end{ex}


\section{Relation between crossed modules and cohomology}

The main theorem of this article concerns the relation between crossed modules and cohomology. Recall that the operads that we consider are supposed to satisfy the hypotheses \ref{HypothesesOperads}.


\begin{thm}\label{main_theorem}
  Let $\CP$ be an operad satisfying \ref{HypothesesOperads}, let $n$ be a natural number $n\geqslant 1$, let $A$ and $M$ be respectively a $\CP$-algebra and an $A$-module, both concentrated in degree $0$. 
  There is a natural isomorphism of abelian groups
  \[ \crmod^n_\CP(A,M)/\sim\quad \cong\,\,\rmH^{n+1}_\CP(A;M). \]
\end{thm}



The proof of a natural bijection instead of a natural isomorphism in this theorem will occupy the rest of this section. We will enhance the bijection to an isomorphism of abelian groups in the next section (see \cref{corGroupBijection}).

\subsection*{1$^{\rm st}$ step:} Construction of a morphism $\crmod^n_\CP(A,M)\to \rmZ^{n+1}_\CP(A,M)$. 

According to Proposition \ref{proposition_crmod}, an $n$-crossed module in $\crmod^n_\CP(A,M)$ consists of a morphism of dg vector spaces $p\colon (B,\mu,\partial)\xrightarrow{\sim} A\oplus M[n]$, 
where $(B,\partial)$ is a dg vector space $0\to B_n\xrightarrow{\partial_n}\cdots\xrightarrow{\partial_1} B_0\to 0$, thus concentrated in degrees $0,\ldots,n$, and $(B,\mu)$ is a $\CP$-algebra such that $p$ induces an isomorphism of $\CP$-algebras
\[ (B_0\,/\,\partial_1(B_1),\mu)\xrightarrow{\cong}(A,\gamma_A) \]
and such that
\[ 
  p|_{\ker(\partial_n)} \colon 
  (\ker(\partial_n),\mu)
  \xrightarrow{\cong} (M[n],\gamma_M) 
\]
is an isomorphism of $A$-modules (with respect to the above $\CP$-algebra structures).  
We will view the $\CP$-algebra structure $(B,\mu)\in\CP\myhyphen{\sf Alg}$ as a morphism $\mu\colon \CP(B)\to B$, or as a morphism $\mu\colon \CP\to\End_B$ interchangeably. 
It is implicit in the preceding definitions that the structure map $\mu$ induces a map on the quotient:
\[
  \begin{tikzcd}
    \CP(B_0) \arrow[r, "\mu"]  \arrow[d,"\CP(p)"']
    & B_0  \arrow[r,"p"] 
    & B_0 / \partial_1B_1 \\    
    \CP\left(B_0 / \partial_1B_1\right) \arrow[rru, "\exists \ \overline{\mu}_0"', bend right=10]
  \end{tikzcd}.
\]

Now, complete the map $p$ to {\it homotopy retract data} $(i, h, p)$, i.e. we call $V\coloneqq A\oplus M[n]$ and we choose a section $i$ (a morphism of dg vector spaces) of $p:(B,\mu,\partial)\to(V,\gamma)$ and a homotopy $h:(B,\mu,\partial_B)\to(B,\mu,\partial_B)$ such that: 
\[
  \id_B-i\,p=\partial_B\,h + h\,\partial_B 
  \quad\mbox{and}\quad 
  p\,i=\id_V \ ;
\]
it is possible to impose this, because $V \cong H_\bullet(B,\partial_B)$. With these notations, we have $\overline{\mu}_0=p\,\mu_0\,i=\gamma_A$.


By the Homotopy Transfer Theorem \ref{HTT}, we obtain by transfer the structure of an $\Omega\CC$-algebra on $V$: there exists a twisting cochain $\overline{\mu} \in \mathrm{Tw}(\scrC ,\End_V)$ i.e. $\overline{\mu}\colon \CC\to\End_V$ of degree $-1$ such that
\[ 
  \partial_V\overline{\mu}+\overline{\mu}\star\overline{\mu}=0.
\] 
Again, $\overline{\mu}$ corresponds to a morphism 
$\overline{\mu}\colon\CC(A\oplus M[n])\to A\oplus M[n]$ of degree $-1$. We claim that the only possibly non-zero components of $\overline{\mu}$ are 
\[
  \overline{\mu}_A\colon  \CC(A)\too A, \quad
  \overline{\mu}_M\colon  \CC(A;M[n])\too M[n] \quad\mbox{and}\quad
  \overline{\mu}_{n+1}\colon  \CC(A)\too M[n].  
\]
This follows from degree considerations: recall that $\scrC$ is supposed to be connected. The largest possible degree-difference is $n$, and these three maps are the only combinations of degree $0$ and degree $n$ elements which give rise to a degree-difference of at most $n$. Notice that the degree of a non-identity cooperation must be $\geq 1$. Using desuspensions, we have the following three maps :
\[\begin{aligned}
  \widetilde{\mu}_A \coloneqq&\,  \overline{\mu}_A \colon \CC(A) \too A \\
  \widetilde{\mu}_M \coloneqq&\, (-1)^n \susp^{-n}\overline{\mu}_M\susp^n \colon \CC(A;M) \too M \\ 
  \widetilde{\mu}_{n+1} \coloneqq&\, \susp^{-n}\overline{\mu}_{n+1}\colon \CC(A) \too M
\end{aligned}\]
By the Homotopy Transfer Theorem, we have by construction
\[
  \widetilde{\mu}_A=p(\gamma_A\circ\alpha)i^{\otimes\bullet}|_A
  \quad\mbox{and}\quad
  \widetilde{\mu}_M= (-1)^n p(\gamma_M\circ\alpha)i^{\otimes\bullet}|_{(A;M)},
\]
so these maps are completely determined by the $\CP$-algebra structure on $A$ and the $A$-module structure on $M$. 
The map $\widetilde{\mu}_{n+1}$ is a map of degree $-(n+1)$ and is viewed as a cochain $\widetilde{\mu}_{n+1}\in(\rmC^{n+1}_\CP(A;M),\partial_\alpha)$. 
%

\begin{lem}
The cochain $\widetilde{\mu}_{n+1}\in(\rmC^{n+1}_\CP(A;M),\partial_\alpha)$ is a cocycle, i.e. $\widetilde{\mu}_{n+1}\in \rmZ^{n+1}_\CP(A;M)$. 
\end{lem}

\begin{proof}
We want to show that $\widetilde{\mu}_{n+1}\in \rmZ^{n+1}_\CP(A;M)$, i.e. $\partial_{\pi_\alpha}(\widetilde{\mu}_{n+1})=0$. By definition of $\partial_{\pi_\alpha}$, see Equation \eqref{definition_differential}, we have 
\begin{equation}  \label{*}
\partial_{\pi_\alpha}(\widetilde{\mu}_{n+1})=\bar{\partial}(\widetilde{\mu}_{n+1})-(-1)^{n+1}\widetilde{\mu}_{n+1}\,(d_1+d_2),
\end{equation}
where $d_1+d_2$ is the differential of $\rmB_\alpha(A)$, the bar construction of the $\scrP$-algebra $A$. Recall that the differential $d_1$ is induced by the differential on $\CC$ and $d_2$ obtained from extending a suitable composition of $\alpha$ and $\gamma_A$.   
Let us compute the three terms in the sum \eqref{*}:
\[
  \begin{tikzcd}
    \bar{\partial}(\widetilde{\mu}_{n+1})\colon\CC(A) 
    \arrow[r,"\triangle_{(1)}"] 
    &\CC\circ_{(1)}\CC(A) \cong \CC(A;\CC(A))
    \arrow[rr, "\alpha\circ(\id_A;\widetilde{\mu}_{n+1})"]
    && \CP(A;M) 
    \arrow[d, "\CP(i)"']
    \arrow[r, "\gamma_M"]
    & M 
    \\ 
    &&& \CP(B_0;B_n) 
    \arrow[r,"\mu"]
    & B_n 
    \arrow[u,"p"']
  \end{tikzcd}  
\]
This description of $\bar{\partial}(\widetilde{\mu}_{n+1})$ is equivalent to 
\[ 
  \bar{\partial}(\widetilde{\mu}_{n+1})\colon \CC(A)
  \xrightarrow{\triangle_{(1)}}
  \CC\left(A;\CC(A)\right)
  \xrightarrow{\CC(A;\widetilde{\mu}_{n+1})} 
  \CC(A;M) 
  \xrightarrow{\widetilde{\mu}_M} 
  M,
\]
so we have 
\[
  \begin{aligned}
    \bar{\partial}(\widetilde{\mu}_{n+1}) 
    & = \widetilde{\mu}_M \; \CC(A; \widetilde{\mu}_{n+1}) \; \triangle_{(1)}\\ 
    & = (-1)^n\susp^{-n}\overline{\mu}_M\susp^n \; \CC(A; \susp^{-n}\overline{\mu}_{n+1}) \; \triangle_{(1)}\\ 
    &= \overline{\mu}_M \; \CC(A; \susp^{-n}\overline{\mu}_{n+1}) \; \triangle_{(1)} \\
    &=(-1)^n \susp^{-n} \overline{\mu}_M \; \CC(A; \overline{\mu}_{n+1}) \; \triangle_{(1)}
  \end{aligned}
\]
Next comes $\widetilde{\mu}_{n+1}\, d_2$:
\[ 
  \widetilde{\mu}_{n+1}\, d_2\colon \CC(A)
  \xrightarrow{\triangle_{(1)}}
  \CC(A;\CC(A))
  \xrightarrow{\CC(\id_A;\alpha(A))}
  \CC(A,\CP(A))
  \xrightarrow{\CC(\id_A;\gamma_A)}
  \CC(A)
  \xrightarrow{\widetilde{\mu}_{n+1}} M,
\]
 which can be expressed as :
\[
    \widetilde{\mu}_{n+1}\, d_2 
     = \widetilde{\mu}_{n+1} \CC(A; \widetilde{\mu}_A) \, \triangle_{(1)}
     = \susp^{-n} \overline{\mu}_{n+1} \CC(A; \overline{\mu}_A) \, \triangle_{(1)}.
\]
These two identifications imply that
\[ 
  \begin{aligned}
  \bar{\partial}(\widetilde{\mu}_{n+1})-(-1)^{n+1}\widetilde{\mu}_{n+1}\, d_2
  & = (-1)^n\susp^{-n}\left(\overline{\mu}_M \; \CC(A; \overline{\mu}_{n+1}) + \overline{\mu}_{n+1} \CC(A; \overline{\mu}_A)\right) \triangle_{(1)} \\
  & = (-1)^n\susp^{-n} \left(\overline{\mu} \star \overline{\mu} \right) |_{\CC_{n+2}(A)}
  \end{aligned}
\]
The third term is $\widetilde{\mu}_{n+1}\, d_1$:
\[
  \widetilde{\mu}_{n+1}\, d_1\colon \CC(A)
  \xrightarrow{d_\CC\circ A} 
  \CC(A)
  \xrightarrow{\widetilde{\mu}_{n+1}}
  M,
\]
then  \[\widetilde{\mu}_{n+1}\, d_1 = \widetilde{\mu}_{n+1} d_\CC = \susp^{-n}\overline{\mu}_{n+1} d_\CC = \partial\overline{\mu} |_{\CC_{n+2}(A)}.\]
In conclusion, we obtain
\[
  \bar{\partial}(\widetilde{\mu}_{n+1})-(-1)^{n+1}\widetilde{\mu}_{n+1}\,(d_1+d_2)
  =
  (-1)^n\susp^{-n}(\partial\overline{\mu}+\overline{\mu}\star\overline{\mu})|_{\CC_{n+2}(A)}=0,
\]
because $\overline{\mu}\in{\rm Tw}(\CC,\End_V)$ is a twisting morphism. 
\end{proof}


\begin{rem}
In the rest of this proof, we will use the same notation for the transferred structure $\overline{\mu}$ and the associated cocycle $\widetilde{\mu}$.
\end{rem}

\subsection*{$2^{\rm nd}$ step:} Passage to the quotient. 

Here we consider two $n$-crossed modules $(B,\mu^B)$ and $(C,\mu^C)$ over $(A,M)$ (i.e. with $(B_0/\partial_1 B_1,\mu_0^B)\cong(A,\gamma_A)$ an isomorphism of $\CP$-algebras induced by $p^B$ and $p^B|_{\ker(\partial^B_n)}:(\ker(\partial_n^B),\mu_n^B)\cong (M[n],\gamma_M)$ an isomorphism of $A$-modules, and the same statement for $C$ instead of $B$) such that 
\[
  \begin{tikzcd}
  (B,\mu^B) \arrow[rr, "\simeq" , "q"'] \ar[dr,"{p^B}"'] & & (C,\mu^C) \arrow[dl,"{p^C}"]  \\
  & V\coloneqq A\oplus M[n] &
  \end{tikzcd}
\]
where $q\colon(B,\mu^B)\to(C,\mu^C)$ is a quasi-isomorphism of $\CP$-algebras over $V=A\oplus M[n]$. Moreover, $(B,\mu^B)$ and $(C,\mu^C)$ are equipped with homotopy retraction data, denoted by $(h^B,i^B)$ and $(h^C,i^C)$ respectively, which satisfy as before
\[
i^B\,p^B-\id_B=d_B\,h^B+h^B\,d_B,\quad\quad p^B\,i^B=\id_V,
\]
and
\[
i^C\,p^C-\id_C=d_C\,h^C+h^C\,d_C,\quad\quad p^C\,i^C=\id_V.
\]
By the commutativity of the diagram, we have $p^B=p^C\,q$ which implies $p^B\,i^B=\id_V=p^C\,q\,i^B$. 

\medskip

Using the HTT (see Theorem \ref{HTT}), we obtain $\Omega\CC$-structures $(V,\widetilde{\mu}^B)$ and $(V,\widetilde{\mu}^C)$ by transfer from $B$ and $C$ respectively, and by the $1^{\rm st}$ step, we obtain two $(n+1)$-cocycles $\widetilde{\mu}_{n+1}^B, \widetilde{\mu}_{n+1}^C\in \rmZ^{n+1}_\CP(A;M)$. 
As described in \cite[Theorem 10.3.6]{Loday:2012}, the $\Omega\CC$-structures on $(B,\mu^B)$ (which is the original $\CP$-algebra structure) and $(V,\widetilde{\mu}^B)$ (which is the transferred structure) are $\infty$-quasi-isomorphic via an $\infty$-quasi-isomorphism $i^B_{\infty}$ with $\infty$-quasi-inverse $p^B_{\infty}$ (see \cite[Theorem 10.4.1]{Loday:2012} for the existence of the $\infty$-quasi-inverse). The same is true for $(C,\mu^C)$ and $(V,\tilde{\mu}^C)$, where we denote the $\infty$-quasi-isomorphism $i^C_\infty$ and its $\infty$-quasi-inverse $p_\infty^C$. 

Using these data, we may define an $\infty$-morphism $Q$ as the following composition:
\[
  \begin{tikzcd}
  (B,\mu^B) \arrow[r,"q"] & (C,\mu^C) \arrow[d,squiggly,"{p^C_\infty}"] \\ 
  (V,\tilde{\mu}^B) \arrow[u,squiggly,"{i^B_\infty}"] \arrow[r,squiggly,"Q"] & (V,\tilde{\mu}^C)
  \end{tikzcd}
\]
Note that the degree $0$ component of $Q$, denoted $Q_0$, is the identity of $V$. Indeed, as $i^B$ is a section of $p^B$, we have $p^B\,i^B=\id_V$, and as $q$ is a morphism over $V$, we have $p^C\,q=p^B$. These two facts imply $Q_0=p^C\,q\,i^B=p^B\,i^B=\id_V$. (Here we use not only the explicit expression of $i^B_\infty$ from the HTT, but also the explicit expression of the $\infty$-quasi-inverse $p_\infty^C$ from \cite[Theorem 10.4.1]{Loday:2012}.) 

\begin{lem}  \label{lemma_cohomologuous}
In the above situation, the two $(n+1)$-cocycles $\widetilde{\mu}_{n+1}^B, \widetilde{\mu}_{n+1}^C\in \rmZ^{n+1}_\CP(A;M)$ are cohomologuous. 
\end{lem}

\begin{proof}
This is due to the fact that $Q$ is an $\infty$-morphism, which is equivalent (by \cite[Theorem 10.2.3]{Loday:2012}; writing $\rhd$ and $\lhd$ for the left- and right actions as in \cite{HLV}) to the equation
\[
\partial Q=Q\rhd\widetilde{\mu}^B-\widetilde{\mu}^C\lhd Q.
\] 
Let us render explicit the three terms of this equation in degree $n+1$. As the differential in $V=A\oplus M[n]$ is trivial, the degree $(n+1)$-part of $\partial Q$ is just the composition
\[
(\partial Q)|_{n+1}\colon\CC_{n+1}
\xrightarrow{d_\CC} \CC_n
\xrightarrow{Q_n} \End_V.
\]
The degree $(n+1)$-part of $Q\lhd\widetilde{\mu}^B$ is given by the composition

\[
 \begin{tikzcd}
  \CC_{n+1} \arrow[r,"{\triangle_{(1)}}"] & (\CC\circ_{(1)}\CC)_{n+1} \arrow[d] 
  \arrow[rr,"{Q\circ\widetilde{\mu}^B}"] && \End_V^V        \\ 
  & \CC_0\circ_{(1)}\CC_{n+1} 
  \arrow[rr,"{Q_0\circ_{(1)}\widetilde{\mu}^B_{n+1}}"] 
  \arrow[d, phantom, "\oplus" description]
  && \End_V^V\circ_{(1)}\End_V     
  \arrow[u,"{{\rm comp}}"] \\
  & \CC_n\circ_{(1)}\CC_1 
  \arrow[rru,"{Q_n\circ_{(1)}\widetilde{\mu}^B_A}"'] & &
 \end{tikzcd}  
\]  
and the degree $(n+1)$-part of $\widetilde{\mu}^C\rhd Q$ is given by the composition
\[
  \begin{tikzcd}
  \CC_{n+1} 
  \arrow[r,"\triangle"] 
  &  (\CC\circ\CC)_{n+1} 
  \arrow[d] 
  \arrow[r,"{\widetilde{\mu}^C\circ Q}"] 
  &   \End_V^V \\ 
  & \CC_1\circ_{(1)}\CC_{n}  
  \arrow[d, phantom, "\oplus" description]
  \arrow[r,"{\widetilde{\mu}^C_M\circ Q_n}"] 
  & \End_V\circ\End_V^V \arrow[u,"{{\rm comp}}"] \\
  & \CC_{n+1}\circ\CC_0 \arrow[ru,"{\widetilde{\mu}^C_{n+1}\circ Q_0}"'] & 
  \end{tikzcd}  
\] 

Observe that while $Q_n:\CC(A)\to M[n]$ is of degree $0$, the map $Q_n:\CC(A)\to M$ 
is of degree $n$. The map $Q_n:\CC_n\to\End_M^A\subset\End_V^V$ is well-defined in $\rmC^n_\CP(A;M)$, because $\widetilde{\mu}^C_A=\widetilde{\mu}^B_A$ and $\widetilde{\mu}^C_M=\widetilde{\mu}^B_M$. 
Let us compare the three terms of $\partial Q=Q\rhd\widetilde{\mu}^B-\widetilde{\mu}^C\lhd Q$ with $\partial_{\pi_\alpha}(Q_n)$. As before, we have 
\[
  \partial_{\pi_\alpha}(Q_n)=\bar{\partial}(Q_n)-(-1)^n Q_n\,(d_1+d_2).
\]
The first factor reads explicitly:
\[
  \begin{tikzcd}
    \bar{\partial}(Q_n):\CC_{n+1} 
    \arrow[r, "\triangle_{(1)}"]
    & (\CC\circ_{(1)}\CC)_{n+1}
    \arrow[r, "\alpha\circ Q_n"] 
    \arrow[d]
    & \CP\circ\End_V^V 
    \arrow[rr,"\gamma_M\circ \id_{\End_V^V}"]
    && \End_V\circ\End_V^V 
    \arrow[r, "\mathrm{comp}"]
    & \End_V^V 
    \\ 
    & \CC_1\circ\CC_n 
    \arrow[rrru,"\widetilde{\mu}_M^C\circ_{(1)}Q_n"']
    & & & &
  \end{tikzcd}  
\]
The second factor reads explicitly:
\[
  (Q_n\circ d_1)|_{n+1}:\CC_{n+1}
  \xrightarrow{d_\CC}
  \CC_n
  \xrightarrow{Q_n}
  \End_V^V.
\]
The third factor reads explicitly:
\[
  \begin{tikzcd}[column sep=4.5ex]
    (Q_n\circ d_2)|_{n+1}\colon \CC_{n+1} 
    \arrow[r,"\triangle_{(1)}"]
    & (\CC\circ_{(1)}\CC)_{n+1}
    \arrow[rr, "\id_\CC\circ_{(1)}\alpha"]
    \arrow[d]
    && \CC\circ_{(1)}\CP  
    \arrow[rr, "Q_n\circ_{(1)}\gamma_A"]
    && \End_V^V\circ\End_V 
    \arrow[r, "\mathrm{comp}"]
    & \End_V^V. \\
    & \CC_n\circ_{(1)}\CC_1
    \arrow[rrrru, "Q_n\circ_{(1)}\widetilde{\mu}_A^B"']
    &&&&&
  \end{tikzcd}
\]
We observe that with respect to the three terms of $\partial Q=Q\rhd\widetilde{\mu}^B-\widetilde{\mu}^C\lhd Q$, only the terms involving $Q_0$ are missing. Therefore $\partial Q=Q\rhd\widetilde{\mu}^B-\widetilde{\mu}^C\lhd Q$ is equivalent to 
\[
  \partial_{\pi_\alpha}(Q_{n})-(Q_0\,{\mu}^B_{n+1}-{\mu}_{n+1}^C\, Q_0)=0.
\]
As $Q_0=\id_V$, this means that the difference between the cocycles $\widetilde{\mu}^B_{n+1}$ and $\widetilde{\mu}^C_{n+1}$ is a coboundary. 
\end{proof}

\begin{rem}
Note that this step implies in particular that the cohomology class $[\widetilde{\mu}_{n+1}]$ associated to a crossed module $(B,\mu,\partial)$ does not depend on the choice of the homotopy retract data $(i,h)$, because two different homotopy retract data for the same crossed module may be viewed as the special case of the above where $C=B$ and $q=\id_B$. 
\end{rem}

\medskip

\noindent{\bf $3^{\rm rd}$ step:}  Construction of a morphism $\rmZ^{n+1}_\CP(A,M)\to\crmod^n_\CP(A,M)$. 

Let us recall that $\CP$, $\CC$ and $\alpha : \CC \to \CP$ satisfy the hypotheses \ref{HypothesesOperads}.  
We use now the Rectification Theorem \ref{RT} in order to rectify the $\CP_\infty$-algebra structure on $V\coloneqq A\oplus M[n]$ (which we obtained via the HTT) to a $\CP$-algebra structure. 
More precisely, given a cocycle $\widetilde{\mu}_{n+1}\in Z^{n+1}_\CP(A,M)$, we view as before $V\coloneqq A\oplus M[n]$ as a $\CP_\infty$-algebra (with structure maps given by $\widetilde{\mu}_A$, $\widetilde{\mu}_M$ and $\widetilde{\mu}_{n+1}$). The Rectification Theorem \ref{RT} assures then that the $\CP_\infty$-algebra $V$ is $\infty$-quasi-isomorphic to the dg $\CP$-algebra $\Omega_\alpha \Bar_\iota V$, where $\alpha:\CC\to\CP$ is the given twisting morphism and $\iota:\CC\to\Omega\CC$ is the canonical twisting morphism.  

We claim that $\Omega_\alpha \Bar_\iota V$ may be truncated in degree $n$ in order to obtain a crossed module $(B,\mu,\partial)$ in $\crmod^n_\CP(A,M)$. 

\begin{defi}[Truncation]
Let $(B,\partial)$ be a positively graded dg vector space and $n\geq 1$. The $n$th truncation of $B$ is the dg vector space of length $(n+1)$:
\[
  \tau_{\leq n}(B)\coloneqq  
  0\to\coker(\partial_{n})
  \xrightarrow{\partial_{n-1}}
  B_{n-1}
  \xrightarrow{\partial_{n-2}}
  \cdots
  \xrightarrow{\partial_1}
  B_1
  \xrightarrow{\partial_0}
  B_0.
\]
\end{defi}

\begin{rem}
The truncation is conceived in such a way that the natural projection from $(B,\partial)$ to $\tau_{\leq n}(B)$ induces an isomorphism in homology with degree $\leq n$. Observe that $\tau_{\leqslant n}(B)$ is the cokernel of the inclusion of the subcomplex $\cdots \xrightarrow{} B_{n+1} \xrightarrow{\partial_n} \mathrm{Im}(\partial_n) \xrightarrow{} 0 \rightarrow \cdots \rightarrow 0$ in $B$. Thus $\tau_{\leqslant n}(B)$ has the universal property of the cokernel. 
\end{rem}

\begin{lem}
Let $(B,\partial,\mu)$ be a positively graded dg $\CP$-algebra and $n\geq 1$. The $n$-th truncation $\tau_{\leq n}(B)$ inherits a canonical $\CP$-algebra structure.  
\end{lem}

\begin{proof} 
Let $(B,\partial,\mu)$ be a dg $\CP$-algebra and $n\geq 1$. The {\it naive $(n+1)$st truncation} $\sigma_{\leq n+1}(B)$ defined by 
\[
  \sigma_{\leq n+1}(B)\coloneqq 0\to B_{n+1}
  \xrightarrow{\partial_{n}}
  B_{n}
  \xrightarrow{\partial_{n-1}}
  \cdots
  \xrightarrow{\partial_1}
  B_1
  \xrightarrow{\partial_0}
  B_0
\]
has a canonical structure of a $\CP$-algebra given by the composition of $\mu:\CP\to\End_B$ with restriction and projection to $\sigma_{n+1}(B)$. As the operad $\CP$ is concentrated in degree $0$, observe that we have 
\[
  \CP(\sigma_{n+1}(B))_{n+1}\cong\CP(B_0;B_{n+1}).
\]
Denote by $\pi:B_n\to\coker(\partial_n)$ the canonical projection. The first square of the following diagram
\[
  \begin{tikzcd}
    \CP(B_0;B_{n+1})
    \arrow[r, "\mu"]
    \arrow[d, "\CP(B_0;\partial_{n})"']
    & B_{n+1}
    \arrow[d,"\partial_n"] 
    \\
    \CP(B_0;B_{n}) 
    \arrow[r, "\mu"]
    \arrow[d, "\CP(B_0;\pi)"']
    &
    B_n
    \arrow[d,"\pi"]
    \\
    \CP(B_0;\coker(\partial_n))
    \arrow[r,"\widetilde{\mu}"]
    &
    \coker(\partial_n)
  \end{tikzcd}  
\]
commutes by definition, and thus the induced map $\tilde{\mu}$ exists. In this way, the $\CP$-algebra structure on $\sigma_{\leq n+1}(B)$ induces a canonical $\CP$-structure on the truncation $\tau_{\leq n}(B)$. 
\end{proof}

Summing up, the Rectification Theorem together with the truncation lead us to the dg $\CP$-algebra
$\tau_{\leq n}(\Omega_\alpha \Bar_\iota V)$. The dg $\CP$-algebra $\tau_{\leq n}(\Omega_\alpha \Bar_\iota V)$ is a crossed module in $\crmod^n_\CP(A,M)$. It is by construction concentrated in degrees $0,\ldots,n$. It is a dg $\CP$-algebra quasi-isomorphic to $V=A\ltimes M[n]$ by the RT. Let us show how to obtain the other ingredients of an $n$-crossed module. In fact, the $\infty$-quasi-isomorphism $i_\infty$ goes in the "wrong" direction:
\[ 
  \begin{tikzcd}
  V=A\ltimes M[n] \arrow[r,squiggly,"i_\infty"] & \tau_{\leq n}(\Omega_\alpha \Bar_\iota V)  
  \end{tikzcd}
\]
The morphism we need is its $\infty$-quasi-inverse $p_\infty$. It is clear that the first component of the morphism $i_\infty$ is just the inclusion of $A$ into $B:=\tau_{\leq n}(\Omega_\alpha \Bar_\iota V)$, and thus the first component $p$ of its $\infty$-quasi-inverse $p_\infty$ corestricts to a morphism of dg $\CP$-algebras $\pi: B\to A$ such that ${\rm proj}_A\,\,p=\pi$. It is less clear that the map $p$, restricted to $\ker(\pi)$, is a morphism of $B$-modules.

\begin{prop}
The map $p$, restricted to $\ker(\pi)$, is a morphism of $B$-modules.
\end{prop}
\begin{proof}
The $\infty$-morphism $p_\infty:B \leadsto V$  corresponds to a map of coalgebras
$\hat{p}:\CC(B)\to\CC(V)$ such that  
\begin{equation}    \label{D}
  \begin{tikzcd} 
  \CC(B) \arrow[r,"\hat{p}"] \arrow[d,"\delta_B"] & \CC(V) \arrow[d,"\delta_V"]  \\
  \CC(B) \arrow[r,"\hat{p}"] & \CC(V) ,  
  \end{tikzcd}
\end{equation}
where $\delta_B$ and $\delta_V$ are the codifferentials expressing the $\CP_\infty$-algebra structures on $B$ and $V$. 
By construction, the composition of $\delta_V$ with the projection onto $V$ is the $\CP_\infty$-algebra structure $\overline{\mu}$ on $V$. 

The map $\overline{\mu}$ from $\CC(V)$ to $V$ has three (possibly non trivial) components: 
\[(\gamma_A+\gamma_{A;M})\,\alpha +\widetilde{\mu}_{n+1}:\CC(A)\oplus\CC(A;M[n])\too A\ltimes M[n],\] 
i.e. the $\CP$-algebra structure on $A$, the $A$-module structure on $M$ and the cocycle $\widetilde{\mu}_{n+1}$. 

In order to select the module structure, we take elements 
$c\otimes b_1\otimes\ldots\otimes b_k\otimes z\in\CC(B)$ for $c\in\CC$, $b_1,\ldots,b_k\in B_0$ and $z \in\ker(\pi)$. 
We want to show the equality
\begin{equation}   \label{***}
  p(\gamma_B(u)(b_1,\ldots,b_k;z)) = \mu_{A;M}(u)(\pi(b_1)\ldots,\pi(b_k),p(z)),
\end{equation}
for all $u\in\CP$, all $b_1,\ldots,b_k\in B_0$ and all $z\in \ker(\pi)$. We claim that this equality comes from the commutativity of the above diagram \eqref{D}. 

The element $c\otimes b_1\otimes\ldots\otimes b_k\otimes z$ is sent by $\hat{p}$ to $\alpha(c)\otimes p(b_1)\otimes\ldots\otimes p(b_k)\otimes p(z)$, and then by the module structure to $\mu_{A;M}(\alpha(c))({\rm proj}\,p(b_1)\otimes\ldots\otimes {\rm proj}\,p(b_k)\otimes p(z))$, where ${\rm proj}:A\oplus M[n]\to A$ is the projection. 
But $\pi={\rm proj}\,\,p$, thus we obtain $\mu_{A;M}(\alpha(c))(\pi(b_1)\ldots,\pi(b_k),p(z))$. 

On the other hand, the element $c\otimes b_1\otimes\ldots\otimes b_k\otimes z$ is sent by $\delta_B$ to 
$\gamma _B(\alpha(c))(b_1,\ldots,b_k,z)$, and then by $p$ to $p(\gamma_B(\alpha(c))(b_1,\ldots,b_k);z)$.
The fact that other components of $p_\infty$ cannot send elements into $M$ shows that both elements must be equal, i.e. we have shown Equation \eqref{***} for elements $u$ of the form $\alpha(c)$ for a $c\in\CC$. 

Notice that the twisting morphism $\alpha:\CC\to\CP$ is not necessarily surjective. But the fact that $H_\bullet(\Omega\CC)\cong\CP$ implies that the $\alpha$-image of $\CC$ generates $\CP$. This is enough to conclude the proof of Equation \eqref{***}. 
\end{proof}

\medskip

\subsection*{\bf $4^{\rm th}$ step:}  Passage to the quotient. 

Here we show that two cocycles which differ by a coboundary give rise to equivalent crossed modules. The two cocycles give rise to $\CP_\infty$-algebra structures $(A\oplus M[n],\mu)$ and $(A\oplus M[n],\mu')$ on $V\coloneqq A\oplus M[n]$. By the Rectification Theorem, the $\CP_\infty$-algebras $(V,\mu)$ and $(V,\mu')$ give rise to dg $\CP$-algebras $B\coloneqq \tau_{\leq n}(\CP\circ\CC(V,\mu))$ and  $B'\coloneqq \tau_{\leq n}(\CP\circ\CC(V,\mu'))$ respectively, which are $\infty$-quasi-isomorphic to $(V,\mu)$ and $(V,\mu')$ respectively via $\infty$-quasi-isomorphisms $\phi:(V,\mu)\to B$ and  $\phi':(V,\mu')\to B'$.  

We want to translate the coboundary, which is the difference of the two cocycles, into an $\infty$-perturbation of the identity map $\id_V:V\to V$, i.e. into an $\infty$-morphism $q$ whose $0$th component is $\id_V$ and where there is one other possibly non-zero component in degree $n$ given by the coboundary. In fact, that this defines indeed an $\infty$-morphism $q$ follows directly from the reasoning of the proof of Lemma \ref{lemma_cohomologuous}: the fact that the coboundary $\partial_{\pi_\alpha}(q_n)$ is the difference between the two cocycles $\zeta^1_{n+1}$ and $\zeta^2_{n+1}$ becomes translated into  $\partial q=q\rhd \zeta^1-\zeta^2\lhd q$, i.e. that $q$ is an $\infty$-morphism according to \cite[Theorem 10.2.3]{Loday:2012}.  

Our situation is now:
\[
  \begin{tikzcd}
  B=\tau_{\leq n}(\CP\circ\CC(V,\mu))  & B'=\tau_{\leq n}(\CP\circ\CC(V,\mu')) \\
(V,\mu) \arrow[u,squiggly,"\phi"] \arrow[r,squiggly,"q"] & (V,\mu') \arrow[u,squiggly,"\phi'"]  
  \end{tikzcd} .
\] 

The $\infty$-quasi-isomorphisms $\phi$ and $\phi'$ have $\infty$-quasi-inverses $\psi$ and $\psi'$ respectively, thus we may define an $\infty$-morphism $Q$ rendering commutative the following diagram:
\[
  \begin{tikzcd}
  B=\tau_{\leq n}(\CP\circ\CC(V,\mu))   \arrow[d,squiggly,"\psi"] \arrow[r,squiggly,"Q"]  & B'=\tau_{\leq n}(\CP\circ\CC(V,\mu')) \\
(V,\mu) \arrow[r,squiggly,"q"] & (V,\mu') \arrow[u,squiggly,"\phi'"]  
  \end{tikzcd} .
\] 
 
By construction, the $\infty$-morphism $Q$ is an $\infty$-quasi-isomorphism. Recall from  \cite[Theorem 11.4.9]{Loday:2012}, $\infty$-quasi-isomorphisms between dg $\CP$-algebras can be reduced to a zig-zag $B\gets \bullet\to B'$ of quasi-isomorphisms of dg $\CP$-algebras. Note that this theorem also holds in our more general setting where $\CP$ is not necessarily a Koszul operad. The proof of this theorem shows that the zig-zag $B\gets \bullet\to B'$ is more precisely
\[
B\xleftarrow{\sim}
\Omega_\alpha \Bar_\iota(B)
\xrightarrow{\sim}
\Omega_\alpha \Bar_\iota(B')
\xrightarrow{\sim}
B'.
\]

This zig-zag induces an equivalence of crossed modules $B\to B'$. Indeed, it suffices to show that the above  $\infty$-quasi-isomorphisms induce morphisms of crossed modules which are the identity on $A$ and $M$, because the equivalence of crossed modules is the equivalence relation generated by these morphisms. 
This is clear first of all for $q:B\leadsto B'$ by construction (as it is a perturbation of the identity on $V=A\oplus M[n]$) and then it is true in each step of the construction of the zig-zag, as $\Bar_\iota(B)=\CC(B)$ is the cofree conilpotent $\CC$-coalgebra on $B$ and $\Omega_\alpha(\Bar_\iota(B))=\CP(\Bar_\iota(B))$ is the free $\CP$-algebra on $\Bar_\iota(B)$.   

\medskip

\subsection*{\bf $5^{\rm th}$ step:}  The two compositions give the identity. 

It is clear that starting with a cocycle, constructing the crossed module and then extracting the cocycle again gives the identity. Conversely, as the crossed module is given by degree reasons essentially by the cocycle map $\tilde{\mu}_{n+1}$, it is clear that starting from a crossed module, extracting the cocycle and then constructing a crossed module also gives the identity. 

This ends the proof of the natural bijection in Theorem \ref{main_theorem}. 


\section{Baer sum of crossed modules}

In this section, we define the Baer sum of crossed modules and show that the bijection of Theorem \ref{main_theorem} is additive and gives thus an isomorphism of abelian groups. 

As before, let $\CP$ be an operad in vector spaces over $\field$, $A$ be a $\CP$-algebra and $M$ be an $A$-module, and let $p:B\to A\oplus M[n]$ be an $n$-crossed module over $(A,M)$. This means
in particular that the postcomposition of $p$ the canonical projection
$A\oplus M[n]\to A$, denoted by $\pi:B\to A$, is a morphism of
$\mathcal{P}$-algebras and that the restriction of $p$ to $\mathrm{Ker}(\pi)$
is a morphism of $B$-modules.

\begin{lem}  \label{lemma_module_structure}
Let $(i,h)$ be transfer data. Then, the restriction of $i$ to $M[n]$ is a
morphism of $B$-modules.
\end{lem}

\begin{proof}
The fact that $(i,h)$ is transfer data means that $i$ is a section of $p$
(in the category of complexes), i.e. $p\,i=\id$, and
\[
  \id_B-i\,p=\partial\,h+h\,\partial.
\]
Let $b_1,\ldots,b_n\in B$ and $m\in M[n]$. We may suppose that the elements
$b_k$ are homogeneous.
We want to show that
$i(\widetilde{\rho}(b_1,\ldots,b_n;m))=\mu(b_1,\ldots,b_n,i(m))$, where
$\widetilde{\rho}$ is the $B$-module structure on $M$ and $\mu$ is the product
in the $\CP$-algebra $B$.
By definition of the $B$-module
structure on $M$, this means
\begin{equation}    \label{action}
i(\rho(\pi(b_1),\ldots,\pi(b_n);m))=\mu(b_1,\ldots,b_n,i(m)),
\end{equation}  
where the action $\rho$ on the LHS is the $A$-module structure on $M$.

\begin{description}[leftmargin = 0em]
\item[\bf $1^{\rm st}$ case] suppose there exists $k\in \llbracket 1, n\rrbracket$ such that $\deg(b_k)\geqslant 1$.
This implies that
$\pi(b_k)=0$, and thus the LHS of (\ref{action}) is zero. On the other hand,
$\deg(i(m))=n$ and thus $\deg(\mu(b_1,\ldots,b_n,i(m)))>n$, which means that
the RHS is zero, too.

\item[\bf $2^{\rm nd}$ case] suppose that for all $k \in\llbracket 1, n \rrbracket$, $\deg(b_k)=0$. 
The homotopy formula for
$i\,p$, applied to $\mu(b_1,\ldots,b_n,i(m))$, reads:
$$i\, p(\mu(b_1,\ldots,b_n,i(m)))=\mu(b_1,\ldots,b_n,i(m)) - \partial\,
h(\mu(b_1,\ldots,b_n,i(m)))-h\,\partial(\mu(b_1,\ldots,b_n,i(m))).$$
The term $h(\mu(b_1,\ldots,b_n,i(m)))$ is zero, because
$\mu(b_1,\ldots,b_n,i(m))$ is of degree $n$
and the homotopy $h$ raises the degree. The term
$\partial(\mu(b_1,\ldots,b_n,i(m)))$ is
zero by the following small computation. The differential $\partial$ is a derivation of
the product in $B$, therefore
$$\partial(\mu(b_1,\ldots,b_n,i(m)))=\sum_{k=1}^n\pm\mu(b_1,\ldots,\partial b_k,
\ldots,b_n,i(m))\pm \mu(b_1,\ldots,b_n,\partial\,i(m)).$$
Here the first sum is zero, because the degree of the $b_k$ is zero and
$\partial$ lowers degree. The last term is also zero, because $i$ is a
morphism of complexes (implying $\partial(i(m))=i(\partial(m))$) and the
differential on the complex $A\ltimes M[n]$ is zero. In conclusion, all terms
involving the homotopy $h$ are zero. The remaining terms read
$$i\, p(\mu(b_1,\ldots,b_n,i(m)))=\mu(b_1,\ldots,b_n,i(m)).$$
The LHS becomes
$i\, p(\mu(b_1,\ldots,b_n,i(m)))=
i(\rho(\pi(b_1),\ldots,\pi(b_n);p\,i(m)))=i(\rho(\pi(b_1),\ldots,\pi(b_n);m))$,
because $p$ is a morphism of $B$-modules and $p\,i=\id$. Thus we have shown
$$i(\rho(\pi(b_1),\ldots,\pi(b_n);m))=\mu(b_1,\ldots,b_n,i(m)),$$
which is \eqref{action}. 
\end{description}
\end{proof}  

Let $q:C\to A\oplus M[n]$ be another $n$-crossed module over $(A,M)$. Then,
let $(i,h)$ and $(j,k)$ be transfer data for $B$ and $C$ respectively and
define $D$ to fit in the following pushout square (in the category of chain
complexes over $A$): 
\[
  \xymatrix{
    (A\oplus M[n])\times_A  (A\oplus M[n])\ar[r]^{\hspace{1.3cm} +}\ar[d]^{i\times j}
    & A\oplus M[n]\ar[d]   \\
    B\times_A C\ar[r] & D
   }
  \]

Observe that the pushout in the category of chain
complexes over $A$ is in fact the usual pushout in the category of chain
complexes:

\begin{lem}
Let
\[
  \xymatrix{
    (\pi_X:X\to A)\ar[r]\ar[d]
    & (\pi_C:C\to A) \ar[d]   \\
    (\pi_B:B\to A) \ar[r] & (\pi_D:D\to A)
   }
  \]
be the pushout in the category of chain complexes over the fixed chain
complex $A$. Then the chain complex $D$ is the pushout in the diagram
\[
  \xymatrix{
    X\ar[r]^\gamma \ar[d]^\beta
    & C \ar[d]   \\
    B \ar[r] & D\,.
   }
  \]
\end{lem}

\begin{proof}
This is an instance of the general theorem stating that a colimit in an
over category is computed as a colimit in the underlying category, see
\cite{nlab}.

In our case, the pushout in the category of chain complexes can be described by
a quotient
$$D=(B\oplus C)/{\rm Im}(\beta,-\gamma),$$
(see \cite[Ch. III, Lemma 1.1]{HS}). The only
difference which could occur between the pushout in chain complexes over $A$
and in usual chain complexes is that in the quotient, one could restrict to
the images of $(\beta,-\gamma)$ whose projections $\pi_C$ and $\pi_B$ to $A$
coincide. But this is automatic in case the pushout square with the
projections to $A$ commutes. So we see also in our special case that the
two pushouts give the same.

Note that the pushout $D$ carries automatically a projection to $A$, induced
by the sum-map $\pi_B+\pi_C:B\oplus C\to A$. This map is zero on
${\rm Im}(\beta,-\gamma)$ and factors thus to a map $\pi_D:D\to A$ which
extends $\pi_B:B\to A$ and $\pi_C:C\to A$.
\end{proof}  

\begin{thm}  \label{Baer_sum}
The $\CP$-algebra structure on $B\times_A C$ induces a
$\CP$-algebra structure on $D$, which turns it into an $n$-crossed
module over $(A,M)$. Moreover, the equivalence class $[D]$ of $D$ does not
depend on the choices of transfer data, nor on the choice of $B$ and $C$ in
their equivalence classes of $n$-crossed modules.
\end{thm}

\begin{proof}
We have already mentioned that the pushout can be computed as a quotient, see
\cite[Ch. III, Lemma 1.1]{HS}). The quotient description is
$$D=((A\ltimes M[n])\oplus (B\times_A C))/{\rm Im}(-(i\times j),+).$$
Let us show that ${\rm Im}(-(i\times j),+)$ is an ideal in the $\CP$-algebra
$(A\ltimes M[n])\oplus (B\times_A C)$ (the product algebra !).
In fact, this follows immediately from
Lemma \ref{lemma_module_structure}.

Indeed, as the pushout is defined in the category of chain complexes over $A$,
it is enough to consider the $B\times C$-module structure, and to show that 
${\rm Im}(-(i\times j),+)$ is a submodule. This is clear for the map $+$, and 
the lemma tells
us that $i\times j$ is a morphism of $B\times C$-modules, thus its image is
a submodule.

The independence of the transfer data follows from the commutative diagram for
two different transfer data $(i,h)$, $(i',h')$ for $B$ and $(j,k)$, $(j',k')$
for $C$ :
\[
  \xymatrix{
    &(A\oplus M[n])\times_A  (A\oplus M[n])\ar[dl]_{i\times j} \ar[dr]^{i'\times j'}
    &   \\
    B\times_A C\ar[rr]^{(i'\times j')(p_B\times p_C)} & & B\times_A C  \,.
   }
\]
This commutative triangle shows that there is a unique arrow from $D$, the
pushout using $(i,j)$, to $D'$, the pushout using $(i',j')$, and vice-versa
by exchanging the roles of $D$ and $D'$.
Thus the two are canonically isomorphic.

The independence of the equivalence class of $D$ from the equivalence classes
of $C$ and $B$ is shown by a similar argument. 
\end{proof}  

\begin{defi}
The {\bf Baer sum} of $[B]$ and $[C]$ is $[D]$ and is denoted by $[B]+[D]$.
\end{defi}

\begin{prop} \label{proposition_additivity}
If the classes of $n$-crossed modules $[B]$ and $[C]$ correspond (under the map constructed in the proof of Theorem \ref{main_theorem}) to $n$-cocycles $c$ and $c'$, then $[B]+[C]$ corresponds to $c+c'$.
\end{prop}

\begin{proof}
The pushout defining the Baer sum of the two $n$-crossed modules $B$ and $C$ in Theorem \ref{Baer_sum}
is a pushout in the slice category over $A$. This implies that the map 
$$A\oplus M[n])\times_A  (A\oplus M[n])\stackrel{+}{\to} A\oplus M[n]$$
is just the addition on the module component for couples having the same $A$-components. Therefore the two cocycles $\tilde{\mu}_B$ and $\tilde{\mu}_C$ become added in the pushout.   
\end{proof}

\begin{cor}\label{corGroupBijection}
The bijective map from Theorem \ref{main_theorem} is a group homomorphism. Moreover, the Baer sum defined in Theorem \ref{Baer_sum} is an abelien group structure on the set of equivalence classes of $n$-crossed modules. 
\end{cor}

\begin{proof}
The first stement follows directly from Proposition \ref{proposition_additivity}. The second statement follows from the first one, because the addiction of cocycles is an abelian group structure on $\rmZ_\CP^n(A;M)$. 
\end{proof}


\end{document}